\documentclass[draft]{imsart}

\RequirePackage[OT1]{fontenc}
\RequirePackage{amsthm,amsmath}
\RequirePackage[numbers]{natbib}
\RequirePackage[colorlinks,citecolor=blue,urlcolor=blue]{hyperref}
\usepackage{color}
\usepackage{amssymb}
\usepackage{amsmath}

\definecolor{ddred}{rgb}{0,0,0}  
\definecolor{dred}{rgb}{0,0,0}  
\definecolor{mred}{rgb}{0, 0, 0}
\definecolor{dgreen}{rgb}{0,0.75,0}
\definecolor{mgreen}{rgb}{0,0.75,0}
\definecolor{dblue}{rgb}{0,0,0.75}
\definecolor{dbblue}{rgb}{0,0,0}
\definecolor{mblue}{rgb}{0,0,0}
\definecolor{dyellow}{rgb}{0.5,0.5,0}
\definecolor{dmagenta}{rgb}{0.75,0,0.75}
\definecolor{dcyan}{rgb}{0,0.5,0.5}
\def\tmred{\textcolor{mred}}
\def\tddred{\textcolor{ddred}}
\def\tdred{\textcolor{dred}}

\startlocaldefs
\numberwithin{equation}{section}
\theoremstyle{plain}
\newtheorem{thm}{Theorem}[section]
\theoremstyle{remark}
\newtheorem{rem}{Remark}[section]
\theoremstyle{plain}
\newtheorem{lem}{Lemma}[section]
\theoremstyle{plain}
\newtheorem{prop}{Proposition}[section]
\endlocaldefs

\begin{document}

\begin{frontmatter}
\title{Universal Scheme for Optimal Search and Stop}
\runtitle{Universal Scheme for Optimal Search and Stop}

\begin{aug}
\author{\fnms{Sirin} \snm{Nitinawarat}
\thanksref{a}
\ead[label=e1]{sirin.nitinawarat@gmail.com}}
\and
\author{\fnms{Venugopal V.} \snm{Veeravalli}
\thanksref{b}
\ead[label=e2]{vvv@illinois.edu}}
%\and
%\author{\fnms{Third} \snm{Author}\thanksref{b}%
%\ead[label=e3]{third@somewhere.com}%
%\ead[label=u1,url]{www.foo.com}}

\address[a]{Qualcomm Technologies, Inc.,
5775 Morehouse Drive,
San Diego, CA 92121,
USA.\\
\printead{e1}
}
\address[b]{Coordinated Science Laboratory,
Department of Electrical and Computer Engineering,
University of Illinois at Urbana-Champaign,
Urbana, IL, 61801,
USA.\\
\printead{e2}}

%\address[b]{Address of the Second and Third author,
%usually few lines long,
%usually few lines long.
%\printead{e3},
%\printead{u1}}

\runauthor{S. Nitinawarat and V.V. Veeravalli}

\affiliation{University of Illinois at Urbana-Champaign}

\end{aug}

\begin{abstract}
The problem of universal search and stop using an adaptive search policy is considered.  When the target location is searched, the observation is distributed according to the target distribution, otherwise it is distributed according to the absence distribution.  A universal sequential scheme for search and stop is proposed using only the knowledge of the absence distribution, and its asymptotic performance is analyzed.  The universal test is shown to yield a vanishing error probability, and to achieve the {\em optimal} reliability when the target is present, universally for every target distribution.  Consequently, it is established that the knowledge of the target distribution is only useful for improving the reliability for detecting a missing target.  It is also shown that a multiplicative gain for the search reliability equal to the number of searched locations is achieved by allowing adaptivity in the search.  
\end{abstract}


\end{frontmatter}

\section{Introduction}
\label{sec-intro}

We study the problem of universal search and stop using an adaptive search policy.  When the target location is searched, the observation is assumed to be distributed according to the target distribution, otherwise it is distributed according to the absence distribution.  We assume that only the absence distribution is known, and the target distribution can be arbitrarily distinct from the absence distribution.  An adaptive search policy specifies the current search location based on the past observations and past search locations.  At the stopping time, 
the target's location is determined or it is decided that it is missing.
%a final decision is made regarding the target location or for that the target is missing.  
The overall goal is to achieve a certain level of accuracy for the final decision using the fewest number of observations.  The results in this paper should be regarded as a contribution to the long-studied area of search theory (see, e.g., \citep{ahls-wege-book-1987, alpe-gal-book-2003, benk-mont-weis-navres-1991, chud-chud-book-1989, koop-book-1980, ston-book-2004}), in particular, searching for a stationary target in discrete time and space with a discrete search effort (cf. \citep{benk-mont-weis-navres-1991}[Subsection 4.2]).

Conceptually, a desirable goal of the search at each location should be to \tddred{determine} if the target is there.  To this end, a universal sequential test for two hypotheses can be used at each location to collect multiple subsequent observations that will eventually lead to a binary outcome that the target is there or not.  To improve reliability for this binary decision at a particular search location, one can use a test that takes more observations at that location.  If we insist on using the mentioned sequential binary test at each location as an ``inner'' test, then it is convenient to select the current search location based on the past binary outcomes of the subsequent binary tests (instead of all the past observational outcomes of all the searches, generally taken multiple times at each of the locations).  With this imposition, the search and stop problem can be conceptually reduced to the problem of constructing an ``outer'' test for the sequential design of such inner experiments.  This intuitive decomposition leads to our proposed universal sequential test for search and stop.

Universal sequential testing for two hypotheses was first considered for certain parametric families of distributions for continuous observation spaces in \citep{cher-berksym-1961, cher-amstat-1965, lai-astat-1988, schw-amstat-1962}, the latest of which employed the concept of time-dependent thresholding.  Here in Subsection \ref{subsec-univ-SebBinTest}, we look at a non-parametric family of distributions for a {\em finite} observation space, for which we propose a universal test using a suitable time-dependent threshold and analyze its performance.  

Sequential design of experiments with a uniform experimental cost was first considered in \citep{cher-amstat-1959, bess-tech-repo-1960} under a certain positivity assumption for the model, which was successfully dispensed with later in \citep{niti-atia-veer-ieeetac-2013, nagh-javi-astat-2013}.  A generalization of the model with a more complicated memory structure for the experimental outcomes and with non-uniform experimental cost was studied in \citep{niti-veer-sqa-2014}.

We show that when the target is present, the proposed universal test based on the aforementioned decomposition yields a vanishing error probability, and achieves the {\em optimal} reliability, in terms of a suitable exponent for the error probability, universally for every target distribution.  Consequently, \tddred{we establish that} the knowledge of the target distribution is only useful for improving \tddred{the} reliability for detecting a missing target.  We also show that a multiplicative gain for the search reliability equal to the number of searched locations is achieved by allowing adaptivity in the search.  

We review the pertinent existing results on universal sequential testing for two hypotheses and sequential design of experiments in Subsections \ref{subsec-univ-SebBinTest} and \ref{subsec-SeqDesExpCost}, respectively.  The general model for universal search and stop is set up in Section \ref{sec-model}.  We present the proposed sequential test for search and stop and state the main result pertaining to its performance in Section \ref{sec-test-and-performance}.

\section{Preliminaries}
\label{sec-prelim}

Throughout the paper, random variables (rvs) are denoted by capital letters, and their realizations are denoted by the corresponding lower-case letters. All rvs are assumed to take values in finite sets, and all logarithms are the natural ones.  \tmred{For a finite set $\mathcal{X},$ and a probability mass function (pmf) $p$ on $\mathcal{X}$ we write $X \sim p$ to denote that the rv $X$ is distributed according to $p.$}

The following technical facts will be useful; their derivations can be found in \citep[Chapter 11]{cove-thom-eit-book-2006}.  Consider random variables $Y^{n} = \left( Y_1, \ldots, Y_n \right)$ which are independent and identically distributed  (i.i.d.) according to a pmf $p$ on $\mathcal{Y},$ i.e., $Y_i \sim p,\ i = 1, \ldots, n.$  Let $y^{n} = \left( y_1, \ldots, y_n \right) \in \mathcal{Y}^{n}$ be a sequence with an empirical distribution $\gamma = \gamma^{(n)}$ on $\mathcal{Y}.$  It follows that the probability of such sequence $y^{n},$ under the i.i.d. assumption according to the pmf $p$, is 
\begin{align} 
p(y^{n}) \ =\  e^{ -n \, \left [ D(\gamma \| p) + H(\gamma) \right ]}, 
\label{eqn-prelim-fact1}
\end{align}
where $D(\gamma \| p)$ and $H(\gamma)$ are the relative entropy of $\gamma$ and $p$, and entropy of $\gamma$, defined as
\begin{align}
D(\gamma \| p) \ \triangleq \ \sum_{y \in \mathcal{Y}} \gamma(y) \log \frac{\gamma(y)}{p(y)}, \nonumber
\end{align}
and
\begin{align}
H(\gamma) \ \triangleq \ - \sum_{y \in \mathcal{Y}} \gamma(y) \log \gamma(y), \nonumber
\end{align}
respectively.
Consequently, it holds that for each $y^n$, the pmf $p$ that maximizes $\,p(y^{n})\,$ is $\,p=\gamma$, and the associated maximal probability of $y^{n}$ is 
\begin{align} 
\gamma(y^{n}) \ = \ e^{\left [ -nH(\gamma) \right ]}.
\label{eqn-prelim-fact2}
\end{align}
Next, for each $n \geq 1,$ the number of all possible empirical distributions from a sequence of length $n$ in $\mathcal{Y}^n$ is upper bounded by $\left( n + 1 \right)^{\vert \mathcal{Y} \vert}.$  In particular, using this last fact, it can be shown that for any $\epsilon > 0,$ it holds that the probability of the i.i.d. sequence $Y^n$ under $p$ satisfies
\begin{align}
\mathbb{P} \left [ D \left( \gamma \| p \right) \geq \epsilon \right ]
\ \leq\  {\left( n + 1 \right)^{\vert \mathcal{Y} \vert}} e^{-n \epsilon}.
\label{eqn-prelim-fact3}
\end{align}

We now review the relevant preliminary results on universal sequential testing for two hypotheses, and model-based sequential design of experiments with varying experimental cost in Subsections \ref{subsec-univ-SebBinTest} and \ref{subsec-SeqDesExpCost}, respectively.  These results will be key to our proposed universal test for search and stop.

\subsection{Universal Sequential Testing for Two Hypotheses}
\label{subsec-univ-SebBinTest} 

Consider sequential testing between the null hypothesis $H_0$ with i.i.d. observations $Y_k \in \mathcal{Y},\  k = 1, 2, \ldots,$ according to a pmf $\pi$ on $\mathcal{Y},$ and  the alternative hypothesis $H_1$ with i.i.d. $Y_k,\ k = 1, 2, \ldots,$ according to a \tddred{pmf} $\mu \neq \pi.$  We assume that only $\pi$ is known, and nothing is known about $\mu,$ i.e., it can be arbitrarily close to $\pi$.  We further assume that both $\mu$ and $\pi$ have full support on $\mathcal{Y}$.

For a threshold parameter $a > 1,$ we shall employ a sequential test defined in terms of the following (Markov) time:
\begin{align}
\tilde{N}^b \ \triangleq\  
\mathop{\rm{argmin}}_{n \geq 1}~  
\left [ n D \left( \gamma \| \pi \right) > 
\left( \log{a} + n^{\frac{2}{3}} +  \vert \mathcal{Y} \vert \log{(n+1)} \right) \right ],
\label{eqn-time-upperthres-bin-test}
\end{align}
where $\gamma$ denotes the empirical distribution of the observation sequence\\ $\left( y_1, \ldots, y_n \right).$  The test stops at this time or $\lfloor a \left( \log{a} \right)^{\rho_1} \rfloor$ for some $\rho_1 > 1,$ depending on which one is smaller, i.e., it stops at time $N^b,$ where
\begin{align}
N^b \ \triangleq\ \min \left( \tilde{N}^b, \lfloor a \left( \log{a} \right)^{\rho_1} \rfloor \right).
\label{eqn-stoptime-bin-test}
\end{align}
Correspondingly, the final decision is made according to 
\begin{align}
\delta_b \left( Y^{N^b} \right) = 
\left \{
\begin{array}{cc}
1 &\mbox{if}\ \tilde{N}^b \leq a \left( \log{a} \right)^{\rho_1}\\
0 &\mbox{if}\ \tilde{N}^b > a \left( \log{a} \right)^{\rho_1}.
\end{array}
\right.
\label{eqn-decision-bin-test}
\end{align}

\begin{lem}  
\label{lem-1}
With $\mu$ and $\pi$ having full support on $\mathcal{Y}$, for every $a > 1,$ the sequential test in (\ref{eqn-time-upperthres-bin-test}), (\ref{eqn-stoptime-bin-test}), (\ref{eqn-decision-bin-test}) yields that
\begin{align}
\alpha_{a}
~\triangleq~ \mathbb{P}_0 \left [ \delta_b \left( Y^{N^b} \right) \ =\ 1 \right ] 
~\leq~ \frac{1}{a}.
\label{eqn-falsealarm-USPRT}
\end{align}
In addition, for any $\nu < 1,$ $\mu \neq \pi$ and every 
$a >  a^* \left(\nu, \mu, \pi \right)$, the test also yields that\\
%{\em \tmred{COMMENT: see if you can set $a^* \left(\nu, \mu, \pi \right)$ independently of $\mu,$ i.e., $a^* \left(\nu, \mu, \pi \right) = a^* \left(\nu, \pi \right).$}}
\tddred{
\begin{align}
\label{eqn-cost-USPRT}
{c}_{a} &~\triangleq~ \mathbb{E}_1 \left [ N^b \right ]
~\leq~\mathbb{E}_1 \left [\tilde{N}^b \right ] 
~\leq~  \frac{\log{a}}{\nu D \left( \mu \| \pi \right)}, \\
\label{eqn-cost-USPRTb}
{\kappa}_{a} &~\triangleq~ \mathbb{E}_0 \left [ N^b \right ] 
~\leq~  a \left( \log{a} \right)^{\rho_1},  \\
\beta_{a}
&~\triangleq~ 
\mathbb{P}_1 \left [ \delta_b \left( Y^{N^b} \right) = 0 \right ]
~=~ 
\mathbb{P}_1 \left [ \tilde{N}^b > a \left( \log{a} \right)^{\rho_1} \right ] 
\label{eqn-cost-USPRT-2}  \\
&~\leq~
\frac{\mathbb{E}_1 \left [ \tilde{N}^b \right ]}{a \left(\log{a}\right)^{\rho_1}}
\nonumber \\
&~\leq~ \frac{1}{\nu D \left( \mu \| \pi \right) a \left(\log{a}\right)^{\left( \rho_1 -1 \right)}}. 
\label{eqn-misdetect-USPRT}
\end{align}
}
\end{lem}
\tddred{The proof of Lemma \ref{lem-1} will be given in Section \ref{pf-lem1}.}

\subsection{Sequential Design of Experiments with Varying Experimental Cost}
\label{subsec-SeqDesExpCost}

Now we turn our attention to another sequential decision-making problem (a model-based one this time).  Consider the problem of sequential design of experiments to facilitate the eventual testing for $H$ hypotheses.  We assume a (conditionally) memoryless model for the outcome conditioned on the currently chosen experiment.  In particular, under the $i$-th hypothesis, $i \in \left \{1, \ldots, H \right \}= [H],$ and conditioned on the current experiment $u_t =u \in \mathcal{U},$ at time $t = 1, 2, \ldots,$ the current outcome of the experiment, denoted by $Z_t,$ is assumed to be conditionally independent of all past outcomes and past experiments $U^{t-1}, Z^{t-1},$ and to be conditionally distributed according to a pmf $p_i^u$ on $\mathcal{Z}.$  There is a cost function $c:[H] \times \mathcal{U} \rightarrow \mathbb{R}^+,$ and the current experiment $u_t$ is assumed to incur a cost of $c\left( i, u_t \right)$ under the $i$-th hypothesis.  \tdred{We assume that for every $i = 1, \ldots, H,\ u \in \mathcal{U},\ z \in \mathcal{Z},\ 
p_i^u \left( z \right) > 0,\ c \left( i, u \right) > 0.$}  A test consists of a adaptive policy $\phi$ that chooses each experiment as a suitable  (possibly randomized) function of past experiments and their outcomes, a stopping time $\tau$, and a final decision rule $\delta$ that outputs a guess of a hypothesis in $[H]$.  The goal is to design a test to optimize the tradeoff between the cost accumulated up to the final decision, \tddred{as measured by} $\sum\limits_{t = 1}^{\tau} c\left(i, U_t \right), i \in [H]$, and the accuracy of the final decision, \tddred{as measured by} $P_{\rm{max}} \triangleq \max\limits_{i=1, \ldots, H} \mathbb{P}_i \left [ \delta \left( Z^{\tau} \right) \neq i \right ]$.  The problem is model-based: all the (conditional) distributions $p_i^u, j \in [H], u \in \mathcal{U},$ and the cost function $c$ are assumed to be known. 

For each hypothesis $i \in [H],$ let
\begin{align}
q_i^*(u) \ \triangleq\ 
\mathop{\rm{argmax}}_{\tddred{q}}~ 
\frac{\min\limits_{j \neq i} \sum\limits_{u} q(u) D\left( p_i^u \| p_j^u \right)}
{\sum\limits_{u} q(u) c(i, u)}.
\label{eqn-ctrl-exploit-DesExp}
\end{align}
Then an asymptotically optimal test can be specified based on these distributions as follows.  At each time $t \geq 1,$ the ML estimate of the true hypothesis $\hat{i}$ can be computed based on past experiments and their outcomes $u^{t-1}, z^{t-1}$ using the model $p_i^u, i \in [H], u \in \mathcal{U}$ (ties are broken arbitrarily).  For $b > 0,$ during the sparse occasions $t = \lfloor e^{b k} \rfloor,\ k = 0, 1, \ldots$, the experiment is selected to explore all possible options in $\mathcal{U}$ in a round-robin manner independently of $\hat{i}$:
\tddred{for $\mathcal{U} = \left \{ u_1, \ldots, u_{\vert \mathcal{U} \vert} \right \},$
\begin{align}
u_t = u_{\left( k ~\mbox{mod}~ \vert \mathcal{U} \vert  \right) ~+~ 1}.
\label{eqn-ctrl-explore-DesExp}
\end{align}}
At all other times, the current (random) experiment is selected as $U_t \sim q^*_{\hat{i}}$.  Denote the joint distribution under the $i$-th hypothesis of all experiments and their outcomes up to time $t$ (induced by the control policy) by $p_i \left( z^t, u^t \right).$  For a threshold $a'> 1,$ the test stops at time $\tau^*$ and decides \tddred{in favor of} the ML hypothesis according to the rule $\delta^*$, where
\begin{align}
\tau^*  \ \triangleq\ \mathop{\rm{argmin}}_{t} 
\min\limits_{j \neq \hat{i}} \frac{p_{\hat{i}} \left( z^t, u^t \right)}{p_j \left( z^t, u^t \right)} \ >\ a',\  \
\delta^* \left( z^{\tau^*}, u^{\tau^*} \right) = \hat{i}.
\label{eqn-stop-DesExp}
\end{align}  
Note that as the $q^*_i,\ i \in [H],$ are, in general, not point-mass distributions, in addition to the realization of all experimental outcomes $z^t,$ we also need to account for the realization of the experiments $u^t$ as well in \tmred{the instantaneous compuation of the ML hypothesis and the stopping criterion (\ref{eqn-stop-DesExp}).  If the experiments have been chosen deterministically at all times, we can just use the joint distributions of all experimental outcomes $p_i \left( z^t \right),\ i \in [H],\ t = 1, 2, \ldots$ in these computations.}  
The resulting test is asymptotically optimal and its performance is characterized in Proposition \ref{prop-1} as follows.\newpage

\begin{prop}[\citep{niti-veer-sqa-2014}]\footnote{The result in \citep{niti-veer-sqa-2014} was proven for the model in which the cost function depends only on the experiment; however, the proof generalizes to the current setting when the cost function also depends on the hypothesis.}
\label{prop-1}
For $b> 0,$ in (\ref{eqn-ctrl-explore-DesExp}) chosen to be sufficiently small, and as $a' \rightarrow \infty,$ the test in (\ref{eqn-ctrl-exploit-DesExp}), (\ref{eqn-ctrl-explore-DesExp}), (\ref{eqn-stop-DesExp}) yields a vanishing error probability $P_{\rm{max}} \rightarrow 0,$ and satisfies for each $i = 1, \ldots, H,$ that
\begin{align}
\mathbb{E}_i \left [
\sum\limits_{t = 1}^{\tau^*} c\left(i, U_t \right) 
\right ]
\ =\ \frac{-\log{P_{\rm{max}}}}
{\max\limits_{q}~ 
\frac{\min\limits_{j \neq i} \sum\limits_{u} q(u) D\left( p_i^u \| p_j^u \right)}
{\sum\limits_{u} q(u) c(i, u)}}\left( 1+ o(1) \right).
\nonumber
\end{align}
In addition, the proposed test is asymptotically optimal in the sense that any sequence of tests $\left( \phi, \tau, \delta \right)$ that achieve $P_{\rm{max}} \rightarrow 0$ must \tddred{satisfy}
\begin{align}
\mathbb{E}_i \left [
\sum\limits_{t = 1}^{\tau} c\left(i, U_t \right) 
\right ]
\ \geq\ \frac{-\log{P_{\rm{max}}}}
{\max\limits_{q}~ 
\frac{\min\limits_{j \neq i} \sum\limits_{u} q(u) D\left( p_i^u \| p_j^u \right)}
{\sum\limits_{u} q(u) c(i, u)}}\left( 1+ o(1) \right),
\nonumber
\end{align}
\tddred{for every $i = 1, \ldots, H.$}
\end{prop} 

\section{Model for Search and Stop}
\label{sec-model}

Consider searching for a single target located in one of the $M$ locations.  At each time $k \geq 1,$ if a location without the target is searched, then the observation $Y_k \in \mathcal{Y}$ is assumed to be conditionally independent of all past observations and past search locations, and to be conditionally distributed according to the \tddred{absence distribution $\pi.$}  The distribution $\pi$ represents pure noise, and we shall assume that this distribution is {\em known} to the searcher.  On the other hand, if the target location is searched, then the observation would be conditionally distributed according to the ``target'' distribution $\mu$ on $\mathcal{Y}$ (and would be conditionally independent of past observations and search locations).  We assume that both $\mu$ and $\pi$ have full support on $\mathcal{Y}$.

We also allow for the possibility of an absent target.  In this latter case, the observations at all locations are distributed according to $\pi.$  Denote the search location at time $k \geq 1$ by $U_k \in [M],$ which is allowed to be any function of all past observations $Y^{k-1} = \left( Y_1, \ldots, Y_{k-1} \right)$ and past search locations $U^{k-1} = \left( U_1, \ldots, U_{k-1} \right).$  

It is interesting to note that the most basic search problem with an overlook probability $\alpha > 0$ (see, e.g., Chapters 4, 5 of \citep{ston-book-2004} and Section 4.2 of \citep{benk-mont-weis-navres-1991}) that is uniform over all locations, corresponds to a special case of our general model wherein $\mathcal{Y} = \left \{0, 1 \right \},\ \mu(0) = \alpha,\ \pi(0) = 1$.  In contrast, our model allows for any general (finite) observation space $\mathcal{Y},$ but assumes that both $\mu$ and $\pi$ have full supports.  The degeneracy in the model for the classic search problem as mentioned affords the construction of a search plan that is more efficient than that for our model (with the assumption of full support).  The main concern for the classic search problem has been to come up with the search plan that is absolutely optimal (non-asymptotically), whereas our main concern is to construct a {\em universal} test that is asymptotically efficient in the regime of vanishing \tddred{error probability.}
%This difference clearly differentiates our current work from previous work related to the classic search problem.
%   
%\tmred{\texttt{TO DO: Contrast this key assumption with most literature in discrete search, wherein $S_{\pi}$ is a singleton.}}  

We seek to design a universal sequential test to search the target (or to decide that it is missing).  Precisely speaking, a test consists of a sequential search policy, a stopping rule and a final decision rule.  The stopping rule defines a random stopping time, denoted by $N,$ which is the number of searches taken until the final decision is made.  At the stopping time, the final decision for the target location is made based on the decision rule $\delta: \mathcal{Y}^N \times [M]^N \rightarrow \left \{0, 1, \ldots, M \right \}$, where the $0$ output corresponds to the final decision for a missing target.  The overall goal is to achieve a certain level of accuracy for the final decision using the fewest number of \tddred{observations, universally} for all $\mu \neq \pi$.

\subsection{Fundamental Performance Limit}

When both $\mu$ and $\pi$ are known, the search and stop problem falls under the umbrella of sequential design of experiments with a uniform experimental cost \citep{cher-amstat-1959}.  In particular, there are $M+1$ hypotheses: $0, 1, \ldots, M,$ where the null (0-th) hypothesis corresponds to the possibility of a missing target.  Each $i$-th hypothesis, $i = 1, \ldots, M$, corresponds to a possible location of the present target.  The experiment set corresponds to $\mathcal{U} = [M];$ and the model $p^u_i (y), i = 0, \ldots, M,$ for sequential design of experiments can be identified as 
\begin{align}
p^u_i = \mu,\ u = i,\ \ \ p^u_i = \pi,&\ u \neq i,\ i = 1, \ldots, M,	\nonumber \\
p^u_0 = \pi,&\ u = 1, \ldots, M.
\label{eqn-search-model-as-exper-design}
\end{align}
Then in this idealistic situation when the probabilistic model (both $\mu$ and $\pi$) is known, by particularizing the characterization of the asymptotically optimal performance in 
\tddred{
Proposition \ref{prop-1} 
%(previously discovered in \cite{cher-amstat-1959, bess-tech-repo-1960, niti-atia-veer-ieeetac-2013, nagh-javi-astat-2013} for the case of a uniform cost) 
to our search and stop problem} using (\ref{eqn-search-model-as-exper-design}) and $c(i, u) = 1,\ \mbox{for~each~} i = 0, \ldots, M,\ \mbox{for~each~} u \in \mathcal{U},$ we get that as the error probability\\ $P_{\rm{max}} = \max\limits_{i = 0, \ldots, M}
\mathbb{P}_i \left [ \delta \left( Y^{N}, U^N \right) \neq i \right ]$ is driven to zero, the optimal \tddred{asymptotes of 
%the expected numbers of searches 
$\mathbb{E}_i [N],\ i = 0, \ldots, M,$} can be characterized as follows.

\begin{prop}  
\label{prop-2}
There exists a sequence of tests to search the target that satisfy $P_{\rm{max}} \rightarrow 0$ and \tddred{yield}
\begin{align}
\mathbb{E}_i[N]  &\ =\   
\left \{
\begin{array}{cc}
\frac{- \log{P_{\rm{max}}}}
{\frac{D\left( \pi \| \mu \right)}{M}} (1+o(1)),\ \ \ \ \ 
i = 0, \\
\\
\frac{-\log{P_{\rm{max}}}}
{D\left( \mu \| \pi \right)} 
(1+o(1)),\ \ \ \ \ 
i = 1, \ldots, M.  
\end{array}
\right.
\label{eqn-ideal-exp}
\end{align}
Furthermore, the asymptotic performance in (\ref{eqn-ideal-exp}) (each \tddred{term} in the denominators) is optimal for every $i = 0, \ldots, M$ simultaneously. 
\end{prop}

Of course, the asymptotic performance in Proposition \ref{prop-2} is idealistic, as it requires the knowledge of $\mu$ (with $\pi$ being already known).  When $\mu$ is not known, since $\mu$ can be arbitrarily close to $\pi,$ this asymptotic performance cannot be achieved universally.  Nevertheless, our main contribution (Theorem \ref{thm-1}) described below shows that one can design a universal test (without the knowledge of $\mu$) that drives the error probability to zero and achieves the optimal exponent of $D \left( \mu \| \pi \right)$ under all the {\em non-null} hypotheses universally for any $\mu \neq \pi$.    

\section{Proposed Universal Scheme for Search and Stop and Its Performance}
\label{sec-test-and-performance}

\subsection{Motivation}

Intuitively speaking, a desirable goal of the search at each location should be to \tddred{determine} if the target is there.  To this end, the universal sequential test for two hypotheses in Subsection \ref{subsec-univ-SebBinTest} can be used at each location to collect multiple subsequent observations that will eventually lead to a binary outcome (say 1 if it is guessed that the target is there, and 0 otherwise).  To improve reliability for this binary decision at a particular search location, one can increase the threshold $a$ in (\ref{eqn-time-upperthres-bin-test}), (\ref{eqn-stoptime-bin-test}), (\ref{eqn-decision-bin-test}) with the cost of taking more observations at that location.

If we \tddred{use} the mentioned sequential binary test at each location as the ``inner'' test, then it is convenient to select the current search location based on the past binary outcomes of the subsequent binary tests (instead of the past $\mathcal{Y}$-ary outcomes of every search, generally taken multiple times at each of the locations).  With this imposition, the search and stop problem can be reduced to a problem of constructing an ``outer'' test for \tddred{the} sequential design of such inner experiments, each of which has \tddred{a binary outcome.}
%(instead of the $\mathcal{Y}$-ary outcomes as originally).

Mathematically speaking, we have reduced the original problem of sequential design of $\mathcal{Y}$-ary-output experiments specified by the (conditional) distributions as in (\ref{eqn-search-model-as-exper-design}) to one of sequential design of binary-output experiments specified as
\begin{align}
\mu_b(0) ~=~ 1- \mu_b(1) ~=~ \beta_a,\ 
\pi_b(1) ~=~ 1- \pi_b(0) ~=~ \alpha_a,
\label{eqn-search-model-as-bin-exper-design-0}
\end{align}
and
\begin{align}
p^u_i = \mu_b,\ u = i,&\ \ \ 
p^u_i = \pi_b,\ u \neq i,\ i = 1, \ldots, M, \nonumber \\
&\ \ \ p^u_0 = \pi_b,\ u = 1, \ldots, M
\label{eqn-search-model-as-bin-exper-design}
\end{align}
where $\alpha_{a}, \beta_{a}$ are as defined in  (\ref{eqn-falsealarm-USPRT}), (\ref{eqn-cost-USPRT-2}).  On the other hand, each binary-output experiment will not have the same cost as for the original $\mathcal{Y}$-ary-output experiment.  In particular, the cost of each binary-output experiment can be specified as
\begin{align}
c(i, u) = c_{a},\  u =i,
\ \ &\ c(i, u) = {\kappa}_{a},\ u \neq i,\ i = 1, \ldots, M,
\nonumber \\
c(0, u) &= {\kappa}_{a},\ u = 1, \ldots, M, 
\label{eqn-search-model-cost-bin-exper}
\end{align}
where $c_{a}, {\kappa}_{a}$ are as defined \tddred{in (\ref{eqn-cost-USPRT}) and (\ref{eqn-cost-USPRTb}), respectively}.
%Let $\mu_b, \pi_b$ denote the binary distributions, where  $\mu_b (0) = 1-\mu_b(1) = \beta_{a},$ and $\pi_b (1) = 1-\pi_b(0) = \alpha_{a}.$ 

There is still a large gap in turning the motivation described above into a ``working'' test for search and stop.  To this end, there are two major challenges.  First, the optimal test for sequential design of experiments in (\ref{eqn-ctrl-exploit-DesExp}), (\ref{eqn-ctrl-explore-DesExp}), (\ref{eqn-stop-DesExp}), achieving the performance stated Proposition \ref{prop-2}, requires precise knowledge of the model.  In contrast, the induced model for sequential design of binary-output experiments in 
(\ref{eqn-search-model-as-bin-exper-design-0}), (\ref{eqn-search-model-as-bin-exper-design}) is a complicated function of the inner threshold $a$ for the sequential binary test \tddred{in (\ref{eqn-time-upperthres-bin-test}), (\ref{eqn-stoptime-bin-test}), (\ref{eqn-decision-bin-test}).  Only} an estimate of this ``true'' induced model is available through the bounds for $\alpha_{a}, \beta_{a}$ stated in (\ref{eqn-falsealarm-USPRT}), (\ref{eqn-misdetect-USPRT}) of Lemma \ref{lem-1}.  Second, as the threshold $a'$ for the optimal test in (\ref{eqn-ctrl-exploit-DesExp}), (\ref{eqn-ctrl-explore-DesExp}), (\ref{eqn-stop-DesExp}) increases, the model for sequential design of experiments in Proposition \ref{prop-2} remains fixed.  \tddred{In contrast,} in \tddred{our proposed test,} the ``outer'' threshold for the test for sequential design of binary-output experiments increases together with the inner threshold $a,$ the latter of which determines the induced model in (\ref{eqn-search-model-as-bin-exper-design-0}), (\ref{eqn-search-model-as-bin-exper-design}).  Consequently, the analysis leading to Proposition \ref{prop-2} does not apply to our proposed test.  Our main technical contribution are precisely, first, to overcome these challenges through the proposed test described below in Subsection \ref{subsec-proposedtest}, employing an outer threshold, which is an appropriate function of the inter threshold, and, second, to provide the analysis for its performance, stated in Theorem \ref{thm-1} below.

\subsection{Proposed Universal Test}
\label{subsec-proposedtest}

As mentioned in the previous subsection, the ``true'' induced model for sequential design of 
\tddred{the} binary-output experiments in (\ref{eqn-search-model-as-bin-exper-design-0}), (\ref{eqn-search-model-as-bin-exper-design}) is a complicated function of the inner threshold $a$ and is not available to us.  Nevertheless, Lemma \ref{lem-1} yields that for $a$ sufficiently large (as a function of $\nu,$ in (\ref{eqn-cost-USPRT}) and $\mu, \pi$) 
\begin{align}
\alpha_{a}, \beta_{a} \leq \frac{1}{a}.  
\label{eqn-more-noisy-cond}
\end{align}
Our idea would be to use a mismatched model defined in terms of $\overline{\mu}_b, \overline{\pi}_b,$ where 
\begin{align}
\overline{\mu}_b(0) ~=~ 1- \overline{\mu}_b(1) ~=~
\overline{\pi}_b(1) ~=~ 1-\overline{\pi}_b(0) ~=~ \frac{1}{a}
\label{eqn-reciprocity-mubar-pibar}
\end{align}
to perform the sequential design of \tddred{the} binary-output experiments.  Specifically, instead of (\ref{eqn-search-model-as-bin-exper-design-0}), (\ref{eqn-search-model-as-bin-exper-design}), consider the following mismatched model for sequential design of binary-output experiments
\begin{align}
\overline{p}^u_i = \overline{\mu}_b,\ u = i,&\ \ \  
\overline{p}^u_i = \overline{\pi}_b,\ u \neq i,
\ i = 1, \ldots, M, \nonumber \\
&\ \ \ \overline{p}^u_0 = \overline{\pi}_b,\ u = 1, \ldots, M.
\label{eqn-search-mismatched-model-as-bin-exper-design}
\end{align}
Heuristically speaking, by (\ref{eqn-more-noisy-cond}), this mismatched model is ``more noisy'' than the true model (for large $a$); hence, the test designed based on this mismatched model should be conservative enough to work well for the true model as well.  This intuition will be proven to be correct.

With the mismatched model specified in (\ref{eqn-search-mismatched-model-as-bin-exper-design}), we can now describe our universal test as follows.  At each time $t \geq 1,$ we compute the \tmred{estimate} of the true hypothesis $\hat{i}$ based on past searched locations and their binary outcomes $u^{t-1}, z^{t-1}$ using the (mismatched) model $~\overline{p}_i^u,  i = 0, \ldots, M, u \in [M]$ in (\ref{eqn-search-mismatched-model-as-bin-exper-design}).  Denote $N(i, 1), N(i, 0),\ i \in [M],$ as the number of times the $i$-th location were searched and the sequential binary test in (\ref{eqn-time-upperthres-bin-test}), (\ref{eqn-stoptime-bin-test}),  (\ref{eqn-decision-bin-test}) decides that the target is there, and that the target is not there, respectively.  By the reciprocity of $\overline{\mu}_b$ and $\overline{\pi}_b,$ in (\ref{eqn-reciprocity-mubar-pibar}), the computation of \tmred{this estimate} can be simplified as
\begin{align}
\hat{i} &\ =\
\left \{
\begin{array}{cc}
\mathop{\rm{argmax}}\limits_{i \in [M]}~N(i, 1)-N(i,0)
&\mbox{if~} \max\limits_{i \in [M]} N(i, 1)-N(i, 0) > 0\\
0	&\mbox{if~} \max\limits_{i \in [M]} N(i, 1)-N(i, 0) \leq 0.
\end{array}
\right.
\label{eqn-mainalg-ML-estimation}
\end{align}    
The \tmred{estimation} in (\ref{eqn-mainalg-ML-estimation}) is quite intuitive, as the difference between the numbers of ``searched-and-found'' and ``searched-and-not-found'' at the $i$-th location: $N(i, 1) - N(i, 0),$
\\$i \in [M],$ should \tddred{approximate the} likelihood that the target is there.  When all these numbers are negative, it is most likely that the target is missing.

For $b > 0,$ during the sparse occasions $t = \lfloor e^{b k} \rfloor,\ k = 0, 1, \ldots$, the experiment is selected to explore all locations in a round-robin manner as
\begin{align}
u_t ~=~ \left(  k ~\mbox{mod}~ M \right) + 1
\label{eqn-ctrl-explore-DesBinExp}
\end{align}
independently of $\hat{i}.$  At all the other times, if $\hat{i} \neq 0,$ we shall search at the $i$-th location, i.e.,
\begin{align}
u_t = \hat{i},\  \mbox{if}~~\hat{i} \neq 0.
\label{eqn-ctrl-exploit-DesBinExp-nonnull}
\end{align}
If $\hat{i} = 0,$ we search among all locations with equal frequency, namely, 
\begin{align}
u_t = \left( i_{t'}  ~\mbox{mod}~ M \right) + 1, 
\label{eqn-ctrl-exploit-DesBinExp-0}
\end{align}
where $i_{t'}$ was the search location at the last time $t' < t$ such that $\hat{i} = 0.$
Denote the joint (mismatched) distribution under the $i$-th hypothesis of all binary searched outcomes up to time $t$ (induced by the above control policy) by $\overline{p}_i \left( z^t \right).$  
The test stops at time $\tau $ and decides \tddred{in favor of} the \tmred{current estimate of the hypothesis} as:
\begin{align}
\tau  \ \triangleq\ \mathop{\rm{argmin}}_{t} 
\left [
\left(
\ \mathop{\min\limits_{j = 0, \ldots, M,}}_{j \neq \hat{i}} 
\ 
\frac{\overline{p}_{\hat{i}} \left( z^t \right)}
{\overline{p}_j \left( z^t \right)} 
\right)
\ >\ e^{a^{\rho_2} \left( \log{a} \right)^{\rho_1}} \right ],\  \
\delta \left( z^{\tau} \right) = \hat{i},
\label{eqn-stop-DesBinExp-original}
\end{align}
where $a$ is the inner threshold for the binary test in (\ref{eqn-time-upperthres-bin-test}), (\ref{eqn-stoptime-bin-test}), (\ref{eqn-decision-bin-test}) and some $\rho_2 > 1$.
\tmred{As clarified at the end of the paragraph preceding Proposition \ref{prop-1}, since the search policy in 
(\ref{eqn-ctrl-explore-DesBinExp}), (\ref{eqn-ctrl-exploit-DesBinExp-nonnull}), (\ref{eqn-ctrl-exploit-DesBinExp-0})
specifies the search location as a deterministic function of the current estimate of the hypothesis at all times, it suffices to work with the joint distribution 
$\overline{p}_i \left( z^t \right)$ instead of $\overline{p}_i \left( z^t , u^t \right),$ i.e., $u^t$ can be written as a deterministic function of $z^t$.}

Using (\ref{eqn-reciprocity-mubar-pibar}), (\ref{eqn-search-mismatched-model-as-bin-exper-design}), we can simplify (\ref{eqn-stop-DesBinExp-original}) as
\begin{align}
\tau ~=~\min \left( \tau', \tau_0 \right),
\label{eqn-stop-DesBinExp}
\end{align}
where
\begin{align}
\tau' &~\triangleq~
\mathop{\rm{argmin}}_{t:\  \hat{i} \neq 0}~
\Bigg [ 
\min \left(
\begin{array}{cc}
\left( N(\hat{i}, 1)-N(\hat{i}, 0) \right), \\
\min\limits_{j \neq \hat{i}} 
\left( \left( N(\hat{i}, 1)-N(\hat{i}, 0) \right) 
- \left( N(j, 1)-N(j, 0) \right) \right)
\end{array}
\right)		\label{eqn-stop-DesBinExp-a} \\
&\hspace{1.2in} >~ 
\frac{a^{\rho_2} 
\left( \log{a} \right)^{\rho_1}}{\log{\left( a -1 \right)}} \Bigg  ]; 
\nonumber \\
\tau_0 &~\triangleq~
\mathop{\rm{argmin}}_{t:\  \hat{i} = 0}~
\left [ \min\limits_{i \in [M]} \left( N(i, 0)-N(i, 1) \right)  ~>~ 
\frac{a^{\rho_2} \left( \log{a} \right)^{\rho_1}}{\log{\left( a -1 \right)}} \right ].
\label{eqn-stop-DesBinExp-b}
\end{align}
Of course, the Markov time in (\ref{eqn-stop-DesBinExp-a}) is reached first when 
$\delta \left( z^{\tau} \right) = \hat{i} \neq 0,$ whereas the other Markov time in (\ref{eqn-stop-DesBinExp-b}) is reached first when $\delta \left( z^{\tau} \right) = \hat{i} = 0$.

Note that the total number of $\mathcal{Y}$-ary-output observations $N$ used to produce the search result is related to the stopping time $\tau$ above as
\begin{align}
N \ =\ \sum\limits_{t=1}^{\tau} N^b_{t},
\nonumber
\end{align}
where each $N^b_t,\ t= 1, \ldots, \tau,$ is the number of observations taken at each location until the sequential test in (\ref{eqn-time-upperthres-bin-test}), (\ref{eqn-stoptime-bin-test}),  (\ref{eqn-decision-bin-test}) produces a binary result $Z_t.$  Consequently, we get from successive uses of the property of conditional expectation and (\ref{eqn-search-model-cost-bin-exper}) that under the true hypothesis $i = 0, \ldots, M,$ it holds that
\begin{align}
\mathbb{E}_i \left [ N \right ] 
&\ =\ \mathbb{E}_i \left [ 
\sum\limits_{t=1}^{\tau-1}
N^b_{t}
~+~
\mathbb{E}_i \left [ 
N^b_{\tau}
~\bigg \vert~ Y^{\left( \sum\limits_{t=1}^{\tau-1} N^b_{t} \right)}
\right ]
\right ]
\nonumber \\
&\ = \
\mathbb{E}_i \left [ 
\sum\limits_{t=1}^{\tau-1}
N^b_{t}
~+~
\mathbb{E}_i \left [ 
N^b_{\tau}
~\bigg \vert~ U_{\tau}
\right ]
\right ]
\nonumber \\
&\ = \
\mathbb{E}_i \left [ 
\sum\limits_{t=1}^{\tau-1}
N^b_{t}
~+~
c \left( i, U_{\tau} \right)
\right ]
\nonumber  \\
&\ = \
\mathbb{E}_i \left [ 
\sum\limits_{t=1}^{\tau}
c \left( i, U_{t} \right)
\right ].
\nonumber 
\end{align}

\subsection{Performance of Proposed Test}
\label{subsec-perfOfTest}

\begin{thm}  
\label{thm-1}  For any $\nu < 1$ in (\ref{eqn-cost-USPRT}) and for $b >0$ used in (\ref{eqn-ctrl-explore-DesBinExp}) chosen to be sufficiently small, as $a \rightarrow \infty$, the test in (\ref{eqn-mainalg-ML-estimation}), (\ref{eqn-ctrl-explore-DesBinExp}), (\ref{eqn-ctrl-exploit-DesBinExp-nonnull}), (\ref{eqn-ctrl-exploit-DesBinExp-0}), (\ref{eqn-stop-DesBinExp-original}) yields a vanishing error probability $P_{\rm{max}} \rightarrow 0$ and also satisfies
\begin{align}
\mathbb{E}_i[N]  \ =\ 
\mathbb{E}_i \left [ \sum\limits_{t=1}^{\tau} c\left( i, U_t \right) \right ]
\ \leq\   
\frac{-\log{P_{\rm{max}}}}
{\nu D\left( \mu \| \pi \right)} (1+o(1)),\ \ \ \ \ 
i = 1, \ldots, M,
\label{eqn-thm-1-assertion}
\end{align}
universally for every $\mu \neq \pi.$ 
\end{thm}
\begin{rem}
Compared to the idealistically optimal performance (when $\mu$ is known) in Proposition \ref{prop-2}, it is interesting to note that our universal test is \tddred{universally asymptotically} optimal, except only when the target is missing.  In other words, the knowledge of the target distribution is only useful in improving reliability for detecting the missing target.  
This consequence of our result is \tddred{directly relevant} in practical settings, wherein the knowledge of the target distribution $\mu$ \tddred{would} be lacking before the target is found.
\end{rem}

\subsection{Comparison with Universal Non-Adaptive Scheme for Search and Stop}

Our main result in Theorem \ref{thm-1} \tddred{illustrates} that one can construct a test with adaptive search policy, using only the knowledge of $\pi,$ that yields a vanishing error probability and achieves the exponent of $D \left( \mu \| \pi \right)$ universally for every $\mu \neq \pi$ when the target is present. A natural question that arises is how much can be gained by employing such an adaptive search policy beyond a non-adaptive one.  A non-adaptive search policy $\overline{\phi}$ has to specify the sequence of search locations at the outset and cannot adapt to the outcomes of the instantaneous searches.  By the symmetry of the problem, there is no reason for a non-adaptive search policy to favor any location.  Consequently, the only non-adaptive search policy that should be considered in the universal setting is the one that search all locations with equal frequency:
\begin{align}
u_k ~=~ \left( k~\mbox{mod}~M \right) + 1,\ k \geq 0.
\label{eqn-nonadaptve-search}
\end{align}
We denote this non-adaptive search policy by $\overline{\phi}^*$.  With this search policy, an efficient universal test has been constructed in \citep{li-niti-veer-asilomar-2014}, which we now describe.  

For each time $k = \ell M,\ \ell = 1, 2, \ldots,$ let $\gamma_i,\ i = 1, \ldots, M$ denote the empirical distribution of the observations when the $i$-th location is searched, namely,\\ $\gamma_i = 
\left( y_{i}, y_{M+i}, \ldots, y_{\left(\ell -1\right)M + i} \right),\ i = 1, \ldots, M$.  Next, \tddred{denote} the estimate of the target location $\hat{i}$ as
\begin{align}
\hat{i} ~=~ \mathop{\rm{argmax}}_{i \in [M]} D \left( \gamma_i \| \pi \right).
\label{eqn-MLest-nonadaptive}
\end{align}
%For a fixed exponent $\overline{\rho},\ 0 < \overline{\rho} < 1,$ and a threshold $\overline{a} > 1,$ and 
With the non-adaptive search policy $\overline{\phi}^*,$ consider the stopping rule defined in terms of the following Markov time:
\begin{align}
\overline{N}' \triangleq
M \times \mathop{\rm{argmin}}_{\ell \geq 1}  \left [ 
\left( D \left( \gamma_{\hat{i}} \| \pi \right) 
-
\max\limits_{j \neq \hat{i}} D \left( \gamma_{j} \| \pi \right) 
\right)
> \log{\overline{a}} 
+ M \vert \mathcal{Y} \vert \log{\left( \ell+1 \right)} \right ].
\label{eqn-stoptime-nonadaptive}
\end{align}
The test stops at time $\overline{N},$ where
\begin{align}
\overline{N} ~\triangleq~ 
\min \left( 
\overline{N}',
\lfloor \overline{a} \log{\overline{a}} \rfloor 
\right).
\label{eqn-stoptime-nonadaptive-2}
\end{align}
Correspondingly, the final decision is made according to 
\begin{align}
\overline{\delta} \left( Y^{\overline{N}} \right) ~=~ 
\left \{
\begin{array}{cc}
\hat{i} &\mbox{if}\ \overline{N}' \leq \overline{a} \log{\overline{a}}\\
0 &\mbox{if}\ \overline{N}' > \overline{a} \log{\overline{a}}.
\end{array}
\right.
\label{eqn-decision-nonadaptive}
\end{align}

The performance of this test with the non-adaptive search scheme follows from the result in 
\citep{li-niti-veer-asilomar-2014}.
\begin{prop}[\citep{li-niti-veer-asilomar-2014}]  
\label{prop-3}  With the non-adaptive search policy $\overline{\phi}^*$ in (\ref{eqn-nonadaptve-search}), the test in (\ref{eqn-MLest-nonadaptive}), (\ref{eqn-stoptime-nonadaptive}), (\ref{eqn-stoptime-nonadaptive-2}), (\ref{eqn-decision-nonadaptive}) yields a vanishing error probability $P_{\rm{max}} \rightarrow 0$ and also satisfies
\begin{align}
\mathbb{E}_i[\overline{N}]  
\ \leq\   
\frac{-\log{P_{\rm{max}}}}
{\frac{D\left( \mu \| \pi \right)}{M}} (1+o(1)),\ \ \ \ \ 
i = 1, \ldots, M,
\nonumber
\end{align}
universally for every $\mu \neq \pi.$ 
\end{prop}
In summary, adaptivity offers a multiplicative gain of $M$ for search reliability beyond non-adaptive searching.  This gain \tddred{increases} with the size of the area to be searched.

%%HERE
\appendix

%\section{Appendix section}\label{app}
\section{}

\subsection{Proof of Lemma \ref{lem-1}}
\label{pf-lem1}

The proof relies on the following lemmas.
\begin{lem}[\citep{li-niti-veer-ieeetit-2014}]
\label{lem-2}
For any pmfs $\mu, \pi$ on $\mathcal{Y}$ with full support, it holds that
\begin{align}
2 B \left( \mu, \pi \right) 
\ =\ \min\limits_{q} D \left( q \| \mu \right) + D \left( q \| \pi \right),
\end{align}
where the minimum above is over all pmfs on $\mathcal{Y}$.
\end{lem}

\begin{lem}  
\label{lem-3}
Under the alternative hypothesis, it holds for every $n \geq 1,$ that
\begin{align}
\mathbb{P}_1 \left [ \tilde{N}^b \geq n \right ]
&\ \leq\ 
a e^{-(n-1) 2 B \left( \mu, \pi \right)} n^{2 \vert \mathcal{Y} \vert}
e^{{\left( n-1 \right)}^{\frac{2}{3}}}.
\nonumber
\end{align}
\end{lem}

\begin{proof}
\begin{align}
\mathbb{P}_1 \left [ \tilde{N}^b \geq n \right ]
&\leq 
\mathbb{P} \left [
(n-1) D \left( \gamma \| \pi \right) 
\ \leq\ 
\log{a} + {\left( n-1 \right)}^{\frac{2}{3}} + \vert \mathcal{Y} \vert \log{(n)}
\right ]	\nonumber \\
&= 
\mathbb{P} 
\left [
\begin{array}{cc}
D \left( \gamma \| \mu \right) 
\geq 
- \frac{\left(\log{a} + {\left( n-1 \right)}^{\frac{2}{3}} 
+ \vert \mathcal{Y} \vert \log{n}\right)}{n-1}  
%\\	
%\hspace{0.8in} 
+ D \left( \gamma \| \mu \right)  + D \left( \gamma \| \pi \right)
\end{array}
\right ]	
\nonumber	\\
&\leq 
\mathbb{P} \left [
D \left( \gamma \| \mu \right) 
\geq 
- \frac{\left(\log{a} + {\left( n-1 \right)}^{\frac{2}{3}} 
+ \vert \mathcal{Y} \vert \log{n}\right)}{n-1}
+ 2 B \left( \mu, \pi \right)
\right ]
\nonumber \\
&
\leq
a e^{-(n-1) 2 B \left( \mu, \pi \right)} 
n^{2\vert \mathcal{Y} \vert}
e^{{\left( n-1 \right)}^{\frac{2}{3}}},
\nonumber
\end{align}
where the second inequality follows from Lemma \ref{lem-2} and the last inequality follows from (\ref{eqn-prelim-fact3}).
\end{proof}

First, we prove (\ref{eqn-falsealarm-USPRT}).  It follows from (\ref{eqn-stoptime-bin-test}), (\ref{eqn-decision-bin-test}) that
\begin{align}
\mathbb{P}_0 \left [ \delta = 1 \right ]
&\ =\ 
\mathbb{P}_0 \left [ N^b = \tilde{N}^b \right ]
\nonumber \\
&\ \leq\
\mathbb{P}_0 \left [ \tilde{N}^b \leq a \left( \log{a} \right)^{\rho_1} \right ]	\nonumber \\
&\ =\ 
\mathbb{P}_0 \left [ \tilde{N}^b \mbox{~is~finite} \right ]	
\nonumber \\
&\ =\ 
\sum\limits_{n=1}^{\infty}
\mathbb{P}_0 \left [ \tilde{N}^b = n \right ]		
\nonumber \\
&\ =\ 
\sum\limits_{n=1}^{\infty}
\mathbb{P}_0 
\left [
n D \left( \gamma \| \pi \right) >
\left( \log{a} + n^{\frac{2}{3}} + \vert \mathcal{Y} \vert \log{(n+1)} \right)			
\right ]		
\nonumber \\
&\ \leq\
\sum\limits_{n=1}^{\infty}
\frac{1}{a} e^{-n^{\frac{2}{3}}}	\label{eqn-Pf-Lem1-1} \\
&\ \leq\ \frac{1}{a},	\nonumber
\end{align}
where (\ref{eqn-Pf-Lem1-1}) follows from (\ref{eqn-prelim-fact2}), (\ref{eqn-prelim-fact3}) and the union bound over the set of all possible empirical distributions.  

It \tddred{now remains} to prove the \tddred{second} inequality in (\ref{eqn-cost-USPRT}), as the equality in (\ref{eqn-cost-USPRT-2}) follows just from the definition of $N^b$ in \tddred{(\ref{eqn-stoptime-bin-test}),} and (\ref{eqn-misdetect-USPRT}) follows from the \tddred{second} inequality in (\ref{eqn-cost-USPRT}).  To this end, it suffices to show that under the alternative hypothesis and as $a \rightarrow \infty,$ 
\begin{align}
\frac{\mathbb{E}_1 \left [\tilde{N}^b \right ]}{\log{a}} \ \rightarrow\ 
\frac{1}{D \left( \mu \| \pi \right)}.
\label{eqn-Pf-Lem1-2} 
\end{align}
First observe that under the alternative hypothesis, $ \| \gamma - \mu \|_1 \rightarrow 0$ a.s.  Since the support of $\mu$ is subsumed in the support of $\pi,\ D \left( \cdot \| \pi \right)$ is continuous in its first argument, and, hence, $D \left( \gamma \| \pi \right) \rightarrow D \left( \mu \| \pi \right)$ a.s.  It then follows from the definition of $\tilde{N}^b$ in (\ref{eqn-time-upperthres-bin-test}) that
\begin{align}
D \left( \gamma^{\tilde{N}^b} \| \pi \right) 
&\ >\ 
\frac{ \log{a} + \left( \tilde{N}^b \right)^{\frac{2}{3}}  
+  \vert \mathcal{Y} \vert \log{\left({\tilde{N}^b}+1 \right)} }
{\tilde{N}^b}
\label{eqn-Pf-Lem1-3}  \\
&\ \leq\ 
\frac{ \log{a} + {\left( \tilde{N}^b - 1 \right)}^{\frac{2}{3}}  
+  \vert \mathcal{Y} \vert \log{\left({\tilde{N}^b} \right)} }
{\tilde{N}^b - 1}.
\label{eqn-Pf-Lem1-4} 
\end{align}
Next, by observing that for any distribution $q,\ D \left( q \| \pi \right) \leq \log 
\left( \frac{1}{\min\limits_{y} \pi(y)} \right),$ we get that 
\begin{align}
\mathbb{P}_1 \left [ \tilde{N}^b \leq n \right ]
&\ \leq\  \mathbb{P}_1 \left [	
n \left( \frac{1}{\min\limits_{y} \pi(y)} \right)  \geq \log{a}
\right ]
\nonumber \\
&\rightarrow\ 0\ \mbox{a.s.,~as}\ a \rightarrow \infty,	\nonumber
\end{align} 
thereby yielding that $\tilde{N}^b \rightarrow \infty$ \tddred{a.s., because, by its definition,}
$\tilde{N}^b$ is non-decreasing.  We now get from this and (\ref{eqn-Pf-Lem1-2}), 
(\ref{eqn-Pf-Lem1-3}), (\ref{eqn-Pf-Lem1-4}) that 
\begin{align}
\frac{\tilde{N}^b}{\log{a}} \ \stackrel{a.s.}{\rightarrow}\ \frac{1}{D \left( \mu \| \pi \right)}.
\label{eqn-Pf-Lem1-5} 
\end{align}

To go from convergence a.s. (\ref{eqn-Pf-Lem1-5}) to convergence in mean (\ref{eqn-Pf-Lem1-2}), it now suffices to prove that the sequence of rvs $\frac{\tilde{N}^b}{\log{a}}$ is uniformly integrable as $a \rightarrow \infty$.  To this end, for any $\eta > 0,$ sufficiently large, we shall upper bound the following quantity using Lemma \ref{lem-3} as follows.
\begin{align}
\mathbb{E}_1
\left [\frac{\tilde{N}^b}{\log{a}}~
\mathbb{I}_{\left \{ \frac{\tilde{N}^b}{\log{T}} \geq \eta \right \}}
\right ]
&\leq 
\mathbb{E}_1
\left [\frac{\left( \tilde{N}^b - \lfloor \eta \log{a} \rfloor + \eta \log{a} \right) }{\log{a}}~
\mathbb{I}_{\left \{ \tilde{N}^b \geq \lfloor \eta \log{a} \rfloor \right \}}
\right ]
\nonumber \\
&\leq 
\frac{1}{\log{a}}
\mathbb{E}_1 \left [
\left(  \tilde{N}^b - \lfloor \eta \log{a}  \rfloor \right)~
\mathbb{I}_{\left \{ \tilde{N}^b - \lfloor \eta \log{a}  \rfloor ~\geq~ 0 \right \}}
\right ] 	\nonumber \\
&\ \ \ \ \ ~+~ \frac{\eta \log{a}}{\log{a}} 
\mathbb{P}_1 
\left [ \tilde{N}^b \geq \lfloor \eta \log{a}  \rfloor \right ]	
\nonumber \\
&= 
\frac{1}{\log{a}} \sum\limits_{l = 1}^{\infty} 
\mathbb{P}_1 
\left [ \tilde{N}^b \geq \lfloor \eta \log{a}  \rfloor + l \right ]	
\nonumber \\
&\ \ \ \ \ ~+~ \eta 
\mathbb{P}_1 
\left [ \tilde{N}^b \geq \lfloor \eta \log{a}  \rfloor \right ]
\nonumber\\
&\leq
\frac{a}{\log{a}} \sum\limits_{l = 1}^{\infty}
\left (
\begin{array}{cc}
e^{
- \left( \eta \log{a} + l - 2 \right) 2 B \left( \mu, \pi \right)
+ \left( \eta \log{a} + l \right)^{\frac{2}{3}} 
}\\ 
\times \left( \lfloor \eta \log{a} \rfloor + l \right)^{2 \vert \mathcal{Y} \vert} 	
\end{array}
\right)
\nonumber \\
& \ \ \ \ \ 
+~ \eta a  
e^{
- \left( \eta \log{a} - 2  \right) 2 B \left( \mu, \pi \right)
+ \left( \eta \log{a}   \right)^{\frac{2}{3}}
} 
\left( \lfloor \eta \log{a} \rfloor \right)^{2 \vert \mathcal{Y} \vert},		\nonumber \\
&\leq
\frac{a}{\log{a}} \sum\limits_{l = 1}^{\infty}
e^{
- \left( \eta \log{a} + l - 4 \right) B \left( \mu, \pi \right)
} 
\left( \lfloor \eta \log{a} \rfloor + l \right)^{2 \vert \mathcal{Y} \vert} 	\nonumber \\
& \ \ \ \ \ 
+~ \eta a  
e^{
- \left( \eta \log{a} - 4  \right) B \left( \mu, \pi \right)
} 
\left( \lfloor \eta \log{a} \rfloor \right)^{2 \vert \mathcal{Y} \vert},
\label{eqn-Pf-Lem1-6}
\end{align}
for any $\eta > \frac{1}{B \left( \mu, \pi \right)}$ and $a$ sufficiently large such that
$(\eta \log{a} ) B (\mu, \pi) \geq \left( \eta \log{a} \right)^{\frac{2}{3}}$.

Continuing from (\ref{eqn-Pf-Lem1-6}), upon noting that for $a$ sufficiently large, it holds that $\lfloor \eta \log{a} \rfloor + l \leq 2 \lfloor \eta \log{a} \rfloor l$, we get
\begin{align}
\mathbb{E}_1
\left [\frac{\tilde{N}^b}{\log{a}}~
\mathbb{I}_{\left \{ \frac{\tilde{N}^b}{\log{a}} \geq \eta 
\right \}}
\right ]
&\ \leq
\frac{a}{\log{a}} \sum\limits_{l = 1}^{\infty}
 e^{- \left( \eta \log{a}  + l - 4 \right) B \left( \mu, \pi \right)} 
\left( 2 \lfloor \eta \log{a}  \rfloor  l \right)^{2 \vert \mathcal{Y} \vert} \nonumber \\
&\ \ \ \ +~ \eta a e^{- \left( \eta \log{a} - 4 \right) B \left( \mu, \pi \right)} 
\left( \lfloor \eta \log{a} \rfloor \right)^{2 \vert \mathcal{Y} \vert}.
\nonumber \\
&\ =\ 
\frac{a}{\log{a}}
\left( \lfloor \eta \log{a}  \rfloor  \right)^{2 \vert \mathcal{Y} \vert}  
e^{- \eta B \left( \mu, \pi \right) \log{a}} 	\nonumber \\
&\ \ \ \ \ \ \times 
\left(
e^{4B \left( \mu, \pi \right)} \sum\limits_{l = 1}^{\infty} 
e^{- B \left( \mu, \pi \right) l}~l^{2 \vert \mathcal{Y} \vert}
\right)
\nonumber \\
&\ \ \ \ \ \  +\ 
\eta a \left( \lfloor \eta \log{a} \rfloor  \right)^{2 \vert \mathcal{Y} \vert}  
e^{- \eta B \left( \mu, \pi \right) \log{a}} 
\times 
e^{4B \left( \mu, \pi \right)},
\nonumber 
\end{align}
which vanishes as $a \rightarrow \infty,$ for any $\eta > \frac{1}{B \left( \mu, \pi \right)},$ thereby establishing the uniform integrability. 

\subsection{Proof of Theorem \ref{thm-1}}

The proof of Theorem \ref{thm-1} relies on the following two lemmas.

\begin{lem}  
\label{lem-4}
For any $\gamma > 2,$ when the parameter $b$ used in (\ref{eqn-ctrl-explore-DesBinExp}) is selected to be sufficiently close to 1, it holds for any true non-null hypothesis $i = 1, \ldots, M,$ and any $\epsilon > 0,$ that the first time $T$ \tddred{from which} the estimate $\hat{i}$ in (\ref{eqn-mainalg-ML-estimation}) always equals the true hypothesis $i,$ satisfies
\begin{align}
%\lim\limits_{n \rightarrow \infty}
\mathbb{P}_i \left [ T > \epsilon t \right ] 
= O \left( t^{-\gamma} \right).
\label{eqn-lem-4}
\end{align}
\end{lem}

\begin{proof}
We first note that for any other hypothesis $j \neq i,\ \ell \geq \lfloor \epsilon t \rfloor,$
\begin{align}
\mathbb{P}_ i 
\left [
\sum\limits_{k=1}^{\ell}
\log{\left( 
\frac{\overline{p}_i^{U_k}\left( Z_k \right)}
{\overline{p}_j^{U_k}\left( Z_k \right)}
\right)} \ \leq\ 0
\right ]
\ \leq\ 
\mathbb{E}_i
\left [
e^{
-\frac{1}{2}
\left(
\sum\limits_{k=1}^{\ell}
\log{\left( 
\frac{\overline{p}_i^{U_k}\left( Z_k \right)}
{\overline{p}_j^{U_k}\left( Z_k \right)}
\right)} 
\right)
}
\right ].
\label{eqn-pf-Lem-4-1}
\end{align}

First, note that for any time $k$ when the estimation in 
(\ref{eqn-mainalg-ML-estimation}) yields $\hat{i} = s \neq i,\ s \neq j,$ we get from (\ref{eqn-reciprocity-mubar-pibar}), 
(\ref{eqn-search-mismatched-model-as-bin-exper-design}) 
that
\begin{align}
\mathbb{E}_i
\left [
e^{
-\frac{1}{2}
\left(
\log{\left( 
\frac{\overline{p}_i^{U_k}\left( Z_k \right)}
{\overline{p}_j^{U_k}\left( Z_k \right)}
\right)} 
\right)
}
\bigg \vert U_k = s
\right ]  
\ =\ 
\mathbb{E}_i
\left [
e^{
-\frac{1}{2}
\left(
\log{1} 
\right)
}
\bigg \vert U_k = s
\right ]  
\ =\ 1\ \mbox{a.s.}
\label{eqn-pf-Lem-4-2}
\end{align}
On the other hand, for the time $k$ when the estimation in (\ref{eqn-mainalg-ML-estimation}) yields $\hat{i} = i,$ or $\hat{i} = j,$ we get from (\ref{eqn-reciprocity-mubar-pibar}), 
(\ref{eqn-search-mismatched-model-as-bin-exper-design})  that
\begin{align}
\mathbb{E}_i
\left [
e^{
-\frac{1}{2}
\left(
\log{\left( 
\frac{\overline{p}_i^{U_k}\left( Z_k \right)}
{\overline{p}_j^{U_k}\left( Z_k \right)}
\right)} 
\right)
}
\bigg \vert U_k = i
\right ]  
&\ =\ 
\beta_a e^{-\frac{1}{2} \log{\left(\frac{1}{a-1}\right)}}
+ \left( 1- \beta_a \right) e^{-\frac{1}{2} \log{\left(a-1\right)}};
\nonumber \\
&\ =\ 
\frac{\beta_a \left( a-1 \right) + \left( 1 - \beta_a \right)}
{\sqrt{\left( a - 1 \right)}}	
\nonumber \\
&\ =\ 
\frac{\beta_a  a +  1 - 2 \beta_a}
{\sqrt{\left( a - 1 \right)}}	
\ \leq\ \frac{2}{\sqrt{\left( a - 1 \right)}}
\ <\ 1;
\label{eqn-pf-Lem-4-3}
\end{align}
\begin{align}
\mathbb{E}_i
\left [
e^{
-\frac{1}{2}
\left(
\log{\left( 
\frac{\overline{p}_i^{U_k}\left( Z_k \right)}
{\overline{p}_j^{U_k}\left( Z_k \right)}
\right)} 
\right)
}
\bigg \vert U_k = j
\right ]  
&\ =\ 
\alpha_a e^{-\frac{1}{2} \log{\left(\frac{1}{a-1}\right)}}
+ \left( 1- \alpha_a \right) e^{-\frac{1}{2} \log{\left(a-1\right)}}
\nonumber \\
&\ \leq\ \frac{2}{\sqrt{\left( a - 1 \right)}} \ <\ 1,
\label{eqn-pf-Lem-4-4}
\end{align}
for $a > 5$, and where the inequalities in (\ref{eqn-pf-Lem-4-3}) and (\ref{eqn-pf-Lem-4-4}) follow from (\ref{eqn-more-noisy-cond}).

Similarly, we get that for any time $t$ when the estimation in 
(\ref{eqn-mainalg-ML-estimation}) yields $\hat{i} = s \neq i,$ we get from (\ref{eqn-reciprocity-mubar-pibar}), 
(\ref{eqn-search-mismatched-model-as-bin-exper-design}) 
that
\begin{align}
\mathbb{E}_i
\left [
e^{
-\frac{1}{2}
\left(
\log{\left( 
\frac{\overline{p}_i^{U_k}\left( Z_k \right)}
{\overline{p}_0^{U_k}\left( Z_k \right)}
\right)} 
\right)
}
\bigg \vert U_k = s
\right ]  
\ =\ 1\ \mbox{a.s.}
\label{eqn-pf-Lem-4-5}
\end{align}
On the other hand, for the time $t$ when the estimation in (\ref{eqn-mainalg-ML-estimation}) yields $\hat{i} = i,$ we get from (\ref{eqn-reciprocity-mubar-pibar}), 
(\ref{eqn-search-mismatched-model-as-bin-exper-design})  that
\begin{align}
\mathbb{E}_i
\left [
e^{
-\frac{1}{2}
\left(
\log{\left( 
\frac{\overline{p}_i^{U_k}\left( Z_k \right)}
{\overline{p}_0^{U_k}\left( Z_k \right)}
\right)} 
\right)
}
\bigg \vert U_k = i
\right ]  
&\ \leq\ \frac{2}{\sqrt{\left( a - 1 \right)}} \ <\ 1,
\label{eqn-pf-Lem-4-6}
\end{align}
for $a > 5.$

Consequently, we get from (\ref{eqn-ctrl-explore-DesBinExp}), 
(\ref{eqn-pf-Lem-4-3}) and (\ref{eqn-pf-Lem-4-4}) that
\begin{align}
\mathbb{P}_i \left [ T > \epsilon t \right ]
&\ \leq\ 
\sum\limits_{\ell = \lfloor \epsilon t \rfloor}^{\infty}
\mathbb{P}_i 
\left [
\overline{p}_i \left( Z^{\ell} \right) \ \leq\ 
\overline{p}_j \left( Z^{\ell} \right) 
\right ] \ +\ 
\mathbb{P}_i 
\left [
\overline{p}_i \left( Z^{\ell} \right) \ \leq\ 
\overline{p}_0 \left( Z^{\ell} \right) 
\right ]
\nonumber \\
&\ \leq\ 
\sum\limits_{\ell = \lfloor \epsilon t \rfloor}^{\infty}
\left( \frac{2}{\sqrt{a-1}} \right)^{\frac{2\log{\ell}}{Mb}}
\ +\ 
\left( \frac{2}{\sqrt{a-1}} \right)^{\frac{\log{\ell}}{M b}}
\nonumber \\
&\ =\  O \left( t^{-\gamma} \right),
\nonumber
\end{align}
for $a > 5$ and for $b$ used in (\ref{eqn-ctrl-explore-DesBinExp}) chosen sufficiently close to 0.
\end{proof}

For $\nu < 1$ and $\rho_1$ as in Lemma \ref{lem-1} and  $a$ sufficiently large, let 
\begin{align}
\overline{c} \left( i, u \right) \ \triangleq\ 
\left \{
\begin{array}{cc}
\frac{\log{a}}{\nu D \left( \mu \| \pi \right)} ~<~a (\log{a} )^{\rho_1}, 
& i = 1, \ldots, M,\ u = i	\\
a (\log{a} )^{\rho_1},	& i = 1, \ldots, M,\  u \neq i.
\end{array}
\right.
\label{eqn-Lem-5-pream-0}
\end{align}
that are the upper bounds for $c_a, \kappa_a,$ as in Lemma \ref{lem-1}, respectively.  Let us consider the ``true'' model for 
sequential design of binary-output experiments specified as 
(\ref{eqn-search-model-as-bin-exper-design-0}),
(\ref{eqn-search-model-as-bin-exper-design}), induced by using the sequential binary test (\ref{eqn-time-upperthres-bin-test}), (\ref{eqn-stoptime-bin-test}), (\ref{eqn-decision-bin-test}) as the ``inner'' test at each location.  Also consider the ``mismatched'' model (for $a$ sufficiently large) as in 
(\ref{eqn-reciprocity-mubar-pibar}), (\ref{eqn-search-mismatched-model-as-bin-exper-design}) satisfying (\ref{eqn-more-noisy-cond}) (cf. Lemma \ref{lem-1}).  In addition, for $i = 1, \ldots, M,$ let
\begin{align}
\overline{d}^a_i
&\ =\ 
\frac{
\mu_b(0) 
\log{\left( \frac{\overline{\mu}_b(0)}{\overline{\pi}_b(0)} \right)}
+ \mu_b(1) 
\log{\left( \frac{\overline{\mu}_b(1)}{\overline{\pi}_b(1)} \right)}
}
{\frac{\log{a}}{\nu D \left( \mu \| \pi \right)}}
\label{eqn-Lem-5-pream} \\
&\ =\ 
\frac{
\beta_a 
\log{\left( \frac{1}{a-1} \right)}
+ \left( 1 - \beta_a \right) 
\log{\left( a-1 \right)}
}
{\frac{\log{a}}{\nu D \left( \mu \| \pi \right)}}
\nonumber \\
&\ \rightarrow\  \nu D \left( \mu \| \pi \right),\  
\mbox{as}~a \rightarrow \infty.
\nonumber 
\end{align}
Then, we have the following lemma.
\begin{lem}
\label{lem-5}
When the causal control policy 
(\ref{eqn-mainalg-ML-estimation}), (\ref{eqn-ctrl-explore-DesBinExp}), (\ref{eqn-ctrl-exploit-DesBinExp-nonnull}), (\ref{eqn-ctrl-exploit-DesBinExp-0}) is applied perpetually, it holds for any true non-null hypothesis $i \in [M],$ any other hypothesis $j = 0, 1, \ldots, M,\ j \neq i$, any small $\epsilon' > 0,$ any $\gamma > 2$ and all $t$ sufficiently large that
\begin{align}
\mathbb{P}_i \left [
\log{\left( \frac{\overline{p}_i \left( Z^t \right)}
{\overline{p}_j \left( Z^t \right)} \right)}
\ <\ 
\left( \sum\limits_{\ell = 1}^t \overline{c} \left( i, U_{\ell} \right) \right) 
\left( \overline{d}^a_i - \epsilon' \right)
\right ]
\ =\ O \left( t^{-\gamma} \right).
\end{align}
\end{lem}
\begin{proof}
We first that with 
\tdred{$\underline{c} = \frac{\log{a}}{\nu D \left( \mu \| \pi \right)},$ and, hence, 
$\overline{c} (i, u) \geq \underline{c},$ we get that}
\begin{align}
\mathbb{P}_i \left [
\log{\left( \frac{\overline{p}_i \left( Z^t \right)}
{\overline{p}_j \left( Z^t \right)} \right)}
\ <\ 
\left( \sum\limits_{\ell = 1}^t \overline{c} \left( i, U_{\ell} \right) \right) 
\left( \overline{d}^a_i - \epsilon' \right)
\right ] 
\hspace{1.3in}
\nonumber 
\end{align}
\begin{align}
&\ \leq\ 
\mathbb{P}_i \left [
\sum\limits_{\ell = 1}^t
\left (
\log{\left( \frac{\overline{p}_i^{U_{\ell}} \left( Z_{\ell} \right)}
{\overline{p}_j^{U_{\ell}} \left( Z_{\ell} \right)} \right)}
~-~
\mathbb{E}_i \left [
\log{\left( \frac{\overline{p}_i^{U_{\ell}} \left( Z_{\ell} \right)}
{\overline{p}_j^{U_{\ell}} \left( Z_{\ell} \right)} \right)}
\bigg \vert
U_{\ell}
\right ]
\right )
\ <\  
-t \underline{c} \frac{\epsilon'}{2}
\right ]
\nonumber \\
&\ \ \ +\ 
\mathbb{P}_i \left [
\sum\limits_{\ell = 1}^t
\left (
\mathbb{E}_i \left [
\log{\left( \frac{\overline{p}_i^{U_{\ell}} \left( Z_{\ell} \right)}
{\overline{p}_j^{U_{\ell}} \left( Z_{\ell} \right)} \right)}
\bigg \vert
U_{\ell}
\right ]
~-~
\overline{d}_i^a \overline{c} \left(i, U_{\ell} \right)
\right )
\ <\ 
-t \underline{c} \frac{\epsilon'}{2}
\right ].
\label{eqn-pf-Lem-5-1}
\end{align}
The proof that the probability of the first term on the right-side of 
(\ref{eqn-pf-Lem-5-1}) goes to zero exponentially fast in $t$ follows from observing that the sequence
\begin{align}
M_t \ =\ 
\sum\limits_{\ell = 1}^t
\left(
\log{\left( \frac{\overline{p}_i^{U_{\ell}} \left( Z_{\ell} \right)}
{\overline{p}_j^{U_{\ell}} \left( Z_{\ell} \right)} \right)}
~-~
\mathbb{E}_i \left [
\log{\left( \frac{\overline{p}_i^{U_{\ell}} \left( Z_{\ell} \right)}
{\overline{p}_j^{U_{\ell}} \left( Z_{\ell} \right)} \right)}
\bigg \vert
U_{\ell}
\right ]
\right)
\nonumber
\end{align}
is a Martingale and can be carried out by invoking the Chernoff bounding argument similar to the argument leading to Equation (5.10) in \citep{cher-amstat-1959}.
In addition, we note that for $a$ sufficiently large,
\begin{align}
\min\limits_{j \neq i}
\min\limits_{k = 1, \ldots, M}
\mathbb{E}_i \left [
\log{\left( \frac{\overline{p}_i^{U_{\ell}} \left( Z_{\ell} \right)}
{\overline{p}_j^{U_{\ell}} \left( Z_{\ell} \right)} \right)}
\bigg \vert
U_{\ell} = k
\right ]
&\ \geq\ 0
\nonumber \\
\max\limits_{u = 1, \ldots, M} \overline{c} \left(i, u \right)
&\ \leq\  a \log{a}^{\rho_1}.
\label{eqn-pf-Lem-5-2}
\end{align}
Next, for the $T$ in Lemma \ref{lem-4}, we get that 
for all $\ell \geq T$ and (\ref{eqn-Lem-5-pream}), (\ref{eqn-Lem-5-pream-0}) that
\begin{align}
\mathbb{E}_i \left [
\log{\left( \frac{\overline{p}_i^{U_{\ell}} \left( Z_{\ell} \right)}
{\overline{p}_j^{U_{\ell}} \left( Z_{\ell} \right)} \right)}
\bigg \vert
U_{\ell}
\right ]
&\ =\
\mathbb{E}_i \left [
\log{\left( \frac{\overline{p}_i^{U_{\ell}} \left( Z_{\ell} \right)}
{\overline{p}_j^{U_{\ell}} \left( Z_{\ell} \right)} \right)}
\bigg \vert
U_{\ell} = i
\right ]	\nonumber \\
&\ =\ 
\overline{d}_i^a
\left( \frac{\log{a}}{\nu D \left( \mu \| \pi \right)} \right)
\nonumber \\
&\ =\ 
\overline{d}_i^a \overline{c} \left(i, U_{\ell} \right).
\label{eqn-pf-Lem-5-3-new}
\end{align}
Consequently, by selecting $\epsilon$ in Lemma \ref{lem-4} sufficiently small, we have that
\begin{align}
\mathbb{P}_i \left [
\sum\limits_{\ell = 1}^t
\left (
\mathbb{E}_i \left [
\log{\left( \frac{\overline{p}_i^{U_{\ell}} \left( Z_{\ell} \right)}
{\overline{p}_j^{U_{\ell}} \left( Z_{\ell} \right)} \right)}
\bigg \vert
U_{\ell}
\right ]
-
\overline{d}_i^a \overline{c} \left(i, U_{\ell} \right)
\right )
<
-t \underline{c} \frac{\epsilon'}{2};\ 
T \leq \epsilon n
\right ] \ =\ 0.
\label{eqn-pf-Lem-5-3}
\end{align}
For such a small $\epsilon,$ we get that the second term on the right-side of (\ref{eqn-pf-Lem-5-1}) can be upper bounded according to
Lemma \ref{lem-4} as
\begin{align}
\mathbb{P}_i \left [
\sum\limits_{\ell = 1}^t
\left (
\mathbb{E}_i \left [
\log{\left( \frac{\overline{p}_i^{U_{\ell}} \left( Z_{\ell} \right)}
{\overline{p}_j^{U_{\ell}} \left( Z_{\ell} \right)} \right)}
\bigg \vert
U_{\ell}
\right ]
-
\overline{d}_i^a \overline{c} \left(i, U_{\ell} \right)
\right )
\ <\ 
-t \underline{c} \frac{\epsilon'}{2}
\right ].
\hspace{0.8in}
\end{align}
\begin{align}
\ \leq\ \mathbb{P}_i \left [ T > \epsilon t \right ]
\ =\  O \left( t^{-\gamma} \right),
\nonumber
\end{align}
thereby completing the proof of Lemma \ref{lem-5}.
\end{proof}

To prove Theorem \ref{thm-1}, we first shall prove that the stopping and final decision rules in (\ref{eqn-stop-DesBinExp-original}) yield that
\begin{align}
P_{\rm{max}} \ \leq\ 
\frac{M}{e^{a^{\rho_2} \left( \log{a} \right)^{\rho_1}}}.
\label{eqn-pf-Thm1-1}
\end{align}
To this end, we consider two separate cases: when the true hypothesis is a non-null hypothesis, and when the true hypothesis is the null hypothesis.  

First consider the case when the true hypothesis is a non-null hypothesis, say $i \in [M].$  For any $t \geq 1$ and a realization $z^t$ of the binary search results of the search policy in (\ref{eqn-ctrl-explore-DesBinExp}), (\ref{eqn-ctrl-exploit-DesBinExp-nonnull}), 
(\ref{eqn-ctrl-exploit-DesBinExp-0}), as in the paragraph preceding (\ref{eqn-mainalg-ML-estimation}), we let $N(i, 1), N(i, 0),\ i \in [M],$ denote the number of times (up to time $t$) that the $i$-th location are searched and the sequential binary test in (\ref{eqn-time-upperthres-bin-test}), (\ref{eqn-stoptime-bin-test}),  (\ref{eqn-decision-bin-test}) decides that the target is there, and that the target is not there, respectively.  Then, we get from 
(\ref{eqn-reciprocity-mubar-pibar}) that for any other non-null hypothesis $j \in [M], j \neq i,$ it holds that
\begin{align}
\frac{\overline{p}_{j} \left( z^t \right)}
{\overline{p}_i \left( z^t \right)} 
\ =\ 
\left( \frac{\frac{1}{a}}{1-\frac{1}{a}} \right)^{N(j, 0)}
\left( \frac{1-\frac{1}{a}}{\frac{1}{a}} \right)^{N(j, 1)}
\left( \frac{1-\frac{1}{a}}{\frac{1}{a}} \right)^{N(i, 0)}
\left( \frac{\frac{1}{a}}{1-\frac{1}{a}} \right)^{N(i, 1)}.
\label{eqn-pf-Thm1-2}
\end{align}
Now, for each such $j \neq i$, consider a pair of distributions 
$\tilde{\mu}_b, \tilde{\pi}_b$ (which are functions of both $i, j$) defined according to 
\begin{align}
\tilde{\mu}_b(0) 
		&\ =\ \left( 1 - \alpha_a \right) 
		\frac{\frac{1}{a}}{1-\frac{1}{a}}
		\ =\ \left( 1 - \pi_b(1) \right) 
		\frac{\frac{1}{a}}{1-\frac{1}{a}}; 	\nonumber \\
\tilde{\pi}_b(1) 	
		&\ =\ \left( 1 - \beta_a \right) 
		\frac{\frac{1}{a}}{1-\frac{1}{a}}
		\ =\ \left( 1 - \mu_b(0) \right) 
		\frac{\frac{1}{a}}{1-\frac{1}{a}}.
\label{eqn-pf-Thm1-3}		
\end{align}
From (\ref{eqn-more-noisy-cond}), and (\ref{eqn-pf-Thm1-3}), we get by an easy calculation that
\begin{align}
\frac{1-\tilde{\mu}_b(0)}{\pi_b(1)} &\ \geq\  
		\frac{1-\frac{1}{a}}{\frac{1}{a}} \ \geq\ 1; 	\nonumber \\
\frac{1-\tilde{\pi}_b(1)}{\mu_b(0)} &\ \geq\ 
		\frac{1-\frac{1}{a}}{\frac{1}{a}} \ \geq\ 1,
\label{eqn-pf-Thm1-4}		
\end{align}
for $a$ large.
Now consider another probability distribution, $\tilde{p}_j\left( z^t \right),$ defined as a function of both $j$ and $i$, based on the same search policy according to
\begin{align}
\tilde{p}_j^u ~=~\tilde{\mu}_b,\ u = j,\ 
\tilde{p}_j^u ~=~\tilde{\pi}_b,\ u = i,\
\tilde{p}_j^u ~=~ \pi_b,\ u \neq i, j.
\label{eqn-pf-Thm1-5}		
\end{align}
Then, it holds that
\begin{align}
\frac{\tilde{p}_{j} \left( z^t \right)}
{p_i \left( z^t \right)} 
&=\ 
\left( \frac{\tilde{\mu}_b(0)}{1-\pi_b(1)} \right)^{N(j, 0)}
\left( \frac{1-\tilde{\mu}_b(0)}{\pi_b(1)} \right)^{N(j, 1)}	
\nonumber \\
&\ \ \ \times
\left( \frac{1-\tilde{\pi}_b(1)}{\mu_b(0)} \right)^{N(i, 0)}
\left( \frac{\tilde{\pi}_b(1)}{1-\mu_b(0)} \right)^{N(i, 1)}.
\label{eqn-pf-Thm1-6}
\end{align}
Consequently, we get from (\ref{eqn-pf-Thm1-2}), (\ref{eqn-pf-Thm1-6}), (\ref{eqn-pf-Thm1-3}) and (\ref{eqn-pf-Thm1-4}) that
\begin{align}
\frac{\tilde{p}_{j} \left( z^t \right)}
{p_i \left( z^t \right)} 
\ \geq\ 
\frac{\overline{p}_{j} \left( z^t \right)}
{\overline{p}_i \left( z^t \right)}.
\label{eqn-pf-Thm1-7}
\end{align}
Similarly, by observing that
\begin{align}
\frac{\overline{p}_{0} \left( z^t \right)}
{\overline{p}_i \left( z^t \right)} 
\ =\ 
\left( \frac{\frac{1}{a}}{1-\frac{1}{a}} \right)^{N(i, 1)}
\left( \frac{1-\frac{1}{a}}{\frac{1}{a}} \right)^{N(i, 0)},
\label{eqn-pf-Thm1-8}
\end{align}
and define $\tilde{p}_0 \left( z^t \right)$ according to
\begin{align}
\tilde{p}_0^u ~=~\tilde{\pi}_b,\ u = i,\ 
\tilde{p}_0^u ~=~\pi_b,\ u \neq i,
\label{eqn-pf-Thm1-9}		
\end{align}
we get from
(\ref{eqn-pf-Thm1-9}), the second equality of (\ref{eqn-pf-Thm1-3}), the second inequality of (\ref{eqn-pf-Thm1-4}) and (\ref{eqn-pf-Thm1-8}) that
\begin{align}
\frac{\tilde{p}_{0} \left( z^t \right)}
{p_i \left( z^t \right)} 
\ =\ 
\left( \frac{\tilde{\pi}_b(1)}{1-\mu_b(0)} \right)^{N(i, 1)}
\left( \frac{1-\tilde{\pi}_b(1)}{\mu_b(0)} \right)^{N(i, 0)}
\ \geq\ \frac{\overline{p}_{0} \left( z^t \right)}
{\overline{p}_i \left( z^t \right)}.
\label{eqn-pf-Thm1-10}
\end{align}

Hence, under the non-null hypothesis $i,$ the error probability incurred by the rules (\ref{eqn-stop-DesBinExp-original}) can be upper bounded as
\begin{align}
\mathbb{P}_i \left [ \delta \left( Z^{\tau} \right) \neq i \right ]
&\ =\ 
	  \mathop{\sum\limits_{j = 0}^M}_{j \neq i}
	  \sum\limits_{t=1}^{\infty}
	  \mathbb{P}_i \left [ \delta \left( Z^{t} \right) = j,\ \tau = t \right ]
	  \nonumber \\
&\ \leq\ 
	  \mathop{\sum\limits_{j = 0}^M}_{j \neq i}
	  \sum\limits_{t=1}^{\infty}
	  \mathbb{P}_i \left [ 
	  	\frac{\tilde{p}_{j} \left( Z^t \right)}{p_i \left( Z^t \right)} 
		> e^{a^{\rho_2} \left( \log{a} \right)^{\rho_1}},\
		\tau = t 
	  \right ]	\nonumber \\
&\ <\ 
	  \mathop{\sum\limits_{j = 0}^M}_{j \neq i}
	  \frac{
	  \sum\limits_{t=1}^{\infty}
	  \tilde{\mathbb{P}}_j \left [ \tau = t \right ]}
	  {e^{a^{\rho_2} \left( \log{a} \right)^{\rho_1}}}		\nonumber \\
&\leq\ 
	  \frac{M}{e^{a^{\rho_2} \left( \log{a} \right)^{\rho_1}}},
	  \label{eqn-pf-Thm1-11}
\end{align}
where the first inequality above follows from the stopping rule in 
(\ref{eqn-stop-DesBinExp-original}), (\ref{eqn-pf-Thm1-7}) and 
(\ref{eqn-pf-Thm1-10}), and the second inequality follows from \tddred{a} change of measure argument.

The error probability under the null hypothesis can be analyzed in a similar manner.  In particular, for any non-null hypothesis $j \in [M],$
we have that
\begin{align}
\frac{\overline{p}_{j} \left( z^t \right)}
{\overline{p}_0 \left( z^t \right)} 
\ =\ 
\left( \frac{\frac{1}{a}}{1-\frac{1}{a}} \right)^{N(j, 0)}
\left( \frac{1-\frac{1}{a}}{\frac{1}{a}} \right)^{N(j, 1)}.
\label{eqn-pf-Thm1-12}
\end{align}
By defining $\tilde{p}_j \left( z^t \right)$ according to
\begin{align}
	\tilde{p}_j^u = \tilde{\mu}_b,\ u = j,\ 
	\tilde{p}_j^u = \pi_b,\ u \neq j,
\end{align}
we get from the first equality of (\ref{eqn-pf-Thm1-3}), the first inequality \tddred{of (\ref{eqn-pf-Thm1-4}), and} (\ref{eqn-pf-Thm1-12})  that
\begin{align}
\frac{\tilde{p}_{j} \left( z^t \right)}
{p_0 \left( z^t \right)} 
\ =\ 
\left( \frac{\tilde{\mu}_b(0)}{1-\pi_b(1)} \right)^{N(j, 0)}
\left( \frac{1-\tilde{\mu}_b(0)}{\pi_b(1)} \right)^{N(j, 1)}
\ \geq\ \frac{\overline{p}_{j} \left( z^t \right)}
{\overline{p}_0 \left( z^t \right)}.
\label{eqn-pf-Thm1-13}
\end{align}
Using (\ref{eqn-pf-Thm1-13}) and the arguments similar to the one leading to (\ref{eqn-pf-Thm1-11}), we get that
\begin{align}
\mathbb{P}_0 \left [ \delta \left( Z^{\tau} \right) \neq 0 \right ]
\ \leq\ 
 \frac{M}{e^{a^{\rho_2} \left( \log{a} \right)^{\rho_1}}},
	  \label{eqn-pf-Thm1-14}
\end{align} 
thereby, together with (\ref{eqn-pf-Thm1-11}),  yielding (\ref{eqn-pf-Thm1-1}).

Next, for a non-null hypothesis $i = 1, \ldots, M,$ and any other hypothesis $j = 0, 1, \ldots, M,\ j \neq i,$ let $\tau_j$ denote the smallest time for which 
\tdred{$\log{\left( \frac{\overline{p}_i \left( Z^t \right)}{\overline{p}_j \left( Z^t \right)} \right)} > a^{\rho_2} \left( \log{a}\right)^{\rho_1}$} for all $t \geq \tau_j.$  Now for any $\delta  > 0$ and $A > 
\frac{a^{\rho_2} \left( \log{a}\right)^{\rho_1}}{\overline{d}_i^a - \delta},$ where $\overline{d}_i^a$ is as in (\ref{eqn-Lem-5-pream}) and with $\overline{\overline{c}} = a \left( \log{a} \right)^{\rho_1},$ it follows that for $a$ large,
\begin{align}
\mathbb{P}_i \left [ 
\sum\limits_{t=1}^{\tau_j} \tdred{\overline{c} \left( i, U_t \right)} > A
\right ]
&\leq
\mathbb{P}_i \left [
\sum\limits_{t=1}^{\tau_j-1} \tdred{\overline{c} \left( i, U_t \right)} 
> A - \overline{\overline{c}};\ 
\left( \tau_j - 1 \right) \geq \lfloor \frac{A}{\overline{\overline{c}}} \rfloor
\right ]
\nonumber \\
&\leq
\mathbb{P}_i \left [
\sum\limits_{t=1}^{\tau_j-1} \tdred{\overline{c} \left( i, U_t \right)}
> \frac{a^{\rho_2} \left( \log{a}\right)^{\rho_1}}
{\left( \overline{d}_i^a - \delta \right)} 
- \overline{\overline{c}};\ 
\left( \tau_j - 1 \right) \geq \lfloor \frac{A}{\overline{\overline{c}}} \rfloor
\right ]
\nonumber \\
&\leq
\sum\limits_{t = \lfloor \frac{A}{\overline{\overline{c}}} \rfloor}^{\infty}
\mathbb{P}_i \left [
\sum\limits_{\ell=1}^{t} \tdred{\overline{c} \left( i, U_{\ell} \right)} 
> \frac{\log{
\left(\frac{\overline{p}_i \left( Z^{t} \right)}
{\overline{p}_j \left( Z^{t} \right)}\right)}
}
{\left( \overline{d}_i^a - \frac{\delta}{2} \right)} 
\right ]
\label{eqn-pf-Thm1-15} \\
&\leq
\sum\limits_{t = \lfloor \frac{A}{\overline{\overline{c}}} \rfloor}^{\infty}
O \left( t^{- \gamma} \right) 
= 
O \left(
\left( \frac{A}{\overline{\overline{c}}} \right)^{-\gamma + 1} 
\right),
\label{eqn-pf-Thm1-16}
\end{align}
where (\ref{eqn-pf-Thm1-15}) follows from the fact that for $a$ large, $\frac{\overline{\overline{c}}}{a^{\rho_2} \left( \log{a}\right)^{\rho_1}} ~\rightarrow 0,$ and (\ref{eqn-pf-Thm1-16}) follows from Lemma \ref{lem-5}.  Consequently, we get from (\ref{eqn-pf-Thm1-16}) that for any $j \neq i,$ 
\begin{align}
\mathbb{E}_i \left [ 
\sum\limits_{t=1}^{\tau_j} \tdred{\overline{c} \left( i, U_t \right)}
\right ]
&\leq 
\frac{a^{\rho_2} \left( \log{a}\right)^{\rho_1}}{\overline{d}_i^a - \delta}
\left( 1 + 
\int_{\frac{a^{\rho_2} \left( \log{a}\right)^{\rho_1}}{\overline{d}_i^a - \delta}}^{\infty}
O \left( \left( \frac{A}{\overline{\overline{c}}} \right)^{-\gamma + 1} \right) dA
\right)
\nonumber \\
&\leq 
\frac{a^{\rho_2} \left( \log{a}\right)^{\rho_1}}{\overline{d}_i^a - \delta}
\left( 
1 + 
\overline{\overline{c}}
O \left( \left( 
\frac{\frac{a^{\rho_2} \left( \log{a}\right)^{\rho_1}}{\overline{d}_i^a - \delta}}{\overline{\overline{c}}} 
\right)^{-\gamma + 2} 
\right)
\right)
\nonumber \\
&\leq 
\frac{a^{\rho_2} \left( \log{a}\right)^{\rho_1}}{\overline{d}_i^a - \delta}
\left( 
1 + 
a \left( \log{a}\right)^{\rho_1}
O \left( \left( 
\frac{a^{\rho_2-1}}{\overline{d}_i^a - \delta}
\right)^{-\gamma + 2} 
\right)
\right)
\nonumber \\
&=
\frac{a^{\rho_2} \left( \log{a}\right)^{\rho_1}}{\overline{d}_i^a - \delta}
\left( 
1+ o(1)
\right),
\label{eqn-pf-Thm1-17}
\end{align}
for $\gamma$ sufficiently large so that $\left(\gamma-2\right) \left( \rho_2 -1 \right) > 1.$

Lastly, it follows from (\ref{eqn-stop-DesBinExp-original}) that $\tau \leq \max\limits_{j \neq i} \tau_j.$  Consequently, we get from (\ref{eqn-pf-Thm1-17}), (\ref{eqn-Lem-5-pream}) \tdred{by virtue of the fact that $\overline{c} \left(i, u \right) \geq c \left( i, u \right)$ (cf. Lemma \ref{lem-1} and (\ref{eqn-Lem-5-pream-0})),} that
\begin{align}
\mathbb{E}_i \left [ 
\sum\limits_{t=1}^{\tau} c \left( i, U_t \right)
\right ]
&\ \leq\ 
\frac{a^{\rho_2} \left( \log{a}\right)^{\rho_1}}{\overline{d}_i^a - \delta}
\left( 
1+ o(1)
\right)
\nonumber \\
&\ =\ 
\frac{a^{\rho_2} \left( \log{a}\right)^{\rho_1}}
{\nu D\left( \mu \| \pi \right) - \delta}
\left( 
1+ o(1)
\right),
\nonumber 
\end{align}
thereby, together with (\ref{eqn-pf-Thm1-1}), yielding (\ref{eqn-thm-1-assertion}) and, hence,
completing the proof, as $\delta$ and $\nu$ can be arbitrarily close to 0 and 1, respectively.

%%HERE
\section*{Acknowledgements}
This work was supported by the Air Force Office of
Scientific Research (AFOSR) under the Grant
FA9550-10-1-0458 through the University of Illinois at
Urbana-Champaign, by the U.S. Defense Threat Reduction
Agency through subcontract 147755 at the University of
Illinois from prime award HDTRA1-10-1-0086, and by the
National Science Foundation under Grant NSF CCF
11-11342.

%\bibliographystyle{acmtrans-ims}
%\bibliography{ims}

\end{document}